\documentclass[reqno,10pt,a4paper]{amsart}
\usepackage{srcltx}
\usepackage{graphicx,color,amssymb,latexsym,amsmath,amsfonts}
\usepackage{mathrsfs}

\numberwithin{equation}{section}
\theoremstyle{plain}
\newtheorem{keyword}{keyword}
\newtheorem{thm}{Theorem}[section]

\newtheorem{lemma}{Lemma}[section]

\newtheorem{proposition}{Proposition}[section]

\begin{document}

\title{On Consistent   Hypothesis Testing}

\
\author{Mikhail Ermakov}\\

\maketitle
 
{erm2512@gmail.com}\\
{Institute of Problems of Mechanical Engineering, RAS and\\
St. Petersburg State University, St. Petersburg, RUSSIA\\
{Mechanical Engineering Problems Institute\\
Russian Academy of Sciences\\
Bolshoy pr.,V.O., 61\\
St.Petersburg\\
Russia}}

\begin{abstract} We establish natural links between different types of consistency: consistency, uniform consistency, pointwise consistency and strong consistency. In particular, we show that the existence of pointwise consistent tests implies the existence of discernible (strongly pointwise consistent)  tests.
    Implementing these results we explore both sufficient conditions and necessary conditions for existence of different types of consistent  tests for the problems of  hypothesis testing on a probability measure of independent sample, on  a mean measure of Poisson process, on a solution of linear ill-posed problem in Gaussian noise, on a solution of deconvolution problem and for a signal detection in Gaussian white noise.
In the last three cases we show that
   necessary conditions and sufficient conditions  coincide.
\end{abstract}

\begin{keyword}[class=AMS]
[Primary ]{62F03},{62G10},{62G20}
\end{keyword}

\begin{keyword}
{hypothesis testing},{consistency},{distinguishability}
\end{keyword}



\section{\bf  Introduction \label{s1}}

The problems of consistent estimation and consistent classification are rather well studied. The universal consistency
 was established for  basic statistical procedures. The necessary and sufficient  conditions  for uniform consistency were obtained. In hypothesis testing the situation is more complicated. The results have a disordered character.   The  paper goal  is

 to  represent a systematic viewpoint to this problem on the base of simple methods and new results,

 to establish natural links between different types of consistency,

 to find both necessary and sufficient conditions of distinquishability of hypotheses  for different types of consistency and for different setups.

    In the works of Berger \cite{be}, Kraft \cite{kr}, Hoefding and Wolfowitz \cite{ho}, Le Cam and Schwartz \cite{les} and Le Cam \cite{le} the problem of existence of consistent tests has obtained a rather deep development. The paper goal is to look at this problem from viewpoint of new research and applications. In many new researches  existence of consistent tests is considered for different classes of problems using special methods. However the consistency of tests as independent statistical problem was not explored systematically.

 In parametric hypothesis testing the consistency of tests is usually evident.
 The consistency of traditional nonparametric tests  (Kolmogorov, $\omega^2$, chi-squared, kernel-based, ...) is also well studied (Lehman and Romano \cite{le}, Balakrishnan,  Nikulin and  Voinov \cite{bnv}, Horowitz and Spokoiny \cite{hosp}, Ermakov \cite{er11} and references therein). For hypothesis testing in functional spaces the uniform consistency   is explored intensively for approaching nonparametric sets of alternatives (Ingster and Suslina \cite{is}, Comminges and  Dalalyan \cite{dal} and references therein). The conditions of distinquishability in these works  satisfy  the conditions of this paper although they have completely another form. In all these setups the hypothesis can be usually considered as a simple.

Interesting problems emerge for nonparametric sets of hypotheses. Such a setup has been explored

 for hypothesis testing that the density has some properties: unimodality, convexity, has a compact support, \ldots (Donoho \cite{do} and Devroye and Lugosi \cite{de}),

 for the problem of classification of probability measures on two classes (Pfanzagl \cite{pf}, Kulkarni and Zeitouni \cite{ku} and Dembo and Peres \cite{dep}),

 for semiparametric hypothesis testing on a sample mean (Bahadur and Savage \cite{bah}, Cover \cite{co} and Dembo and Peres \cite{dep}) and for a more complicated statistical functionals such as  Fisher information (Donoho \cite{do} and Devroye and Lugosi \cite{de}),

 for decernibility of  probability measures  of ergodic  processes (Nobel \cite{no}).

 Different approaches are implemented to the study of consistency of tests. The consistency of tests is explored as usual consistency, uniform consistency, pointwise consistency or discernibility (strong consistency). A sequence of tests is called discernible (Dembo and Peres \cite{dep}, Devroye and Lugosi \cite{de})   if the sequence of tests makes almost surely only finitely number of errors.

 We show that there are natural links between these types of consistency.
 The existence of pointwise consistent  tests implies the existence of discernible  tests.
The existence of  consistent  tests implies that the set of alternatives can  be represented as a countable union of nested subsets  such that, for these subsets,   there are uniformly consistent tests. A similar statement was established also for  pointwise consistent tests.

 Thus the problems of description of sets of hypotheses admitting the consistent or discernible tests are reduced to similar problems for  uniformly consistent tests.

On the base of these statements we explore  both sufficient conditions and necessary conditions for existence of uniformly consistent,
 consistent, pointwise consistent and discernible  tests in
  the problems of hypotheses testing

--  on a probability measure of i.i.d.r.v.'s,
  
-- on a value of statistical statistical functional depending on a probability measure of i.i.d.r.v.'s,

-- on a mean measure of Poisson process,

-- on covariance operator of i.i.d. Gaussian vectors in Hilbert space,

-- on a solution of linear ill-posed problem in Hilbert space if the  noise is Gaussian,

--  on a solution of deconvolution problem,

-- for signal detection in Gaussian white noise.

  For  signal detection in Gaussian white noise, hypothesis testing on a solution of linear ill-posed problem in Gaussian noise and on a solution of deconvolution problem, we find necessary and sufficient conditions for existence of  uniformly consistent tests if the sets of hypotheses and alternatives are bounded. If the sets of hypotheses is bounded, the necessary and sufficient conditions for existence of  consistent tests are established as well.

 For  hypothesis testing on a probability measure of i.i.d.r.v.'s  the consistency of tests were explored in many works for different setups. In the paper the main attention  focuses on the necessary conditions.   Le Cam and Schwartz \cite{les} established necessary and sufficient conditions for  existence of uniformly consistent estimators.  The conditions were  provided in terms of  weak topology (the $\tau$-topology) of setwise convergence on all $n$-fold products  of probability measures and  are not readily verifiable. Hypothesis testing is a particular case of this setup.
 In the paper we discuss one dimensional version of  Le Cam and Schwartz conditions. If the sets of densities of probability measures of hypotheses and alternatives are  uniformly integrable  for some probability measure, the necessary and sufficient conditions for existence of uniformly consistent tests (distinguishability of hypotheses) can be provided in a more simple form.
 The sets of probability measures of hypothesis and alternatives are distinquishable iff they have disjoint closures in the weak topology of setwise convergence (see Theorem \ref{b7}). By wellknown theorems of functional analysis, this form of distinquishability conditions straightforwardly follows from the original Le Cam - Schwartz Theorem.  Note that the Le Cam and Schwartz Theorem \cite{les} is not mentioned in the most  of subsequent works on existence of consistent tests.
 
 The condition of uniform integrability of densities  is satisfied  for all considered recently problems of distinguishability of hypotheses with approaching sets of  alternatives (see Ingster and Suslina \cite{is}, Comminges and  Dalalyan \cite{dal} and references therein) as well as for the most part of semiparametric setups of problems (see Donoho \cite{do} and  Devroye and Lugosi \cite{de}). At the same time the traditional tests of nonparametric hypothesis testing cover all possible sets of probability measures. To cover such a setups we explore in subsection 4.3 the necessary and sufficient conditions of distinguishability for the problem of hypothesis testing on a functional value. All functionals based on distance method  
  as well as a wide class of functionals having Hadamard derivative satisfy the proposed conditions. Note that the sets of hypotheses and alternatives admitting the assignment in a simple form can be separated usually a finite number of functionals. Earlier necessary and sufficient conditions of distinguishability of functionals has been obtained only for the sample mean  (Cover \cite{co}, Dembo and Peres \cite{dep}).
 
 Le Cam and Schwartz Theorem \cite{les} allows to obtain  compact necessary and sufficient conditions of distinguishability of hypotheses on  the probability measure of i.i.d. Gaussian vectors in Hilbert space. These conditions are based on theorems on weak convergence of Gaussian measures (see Bogachev \cite{bog}).  Some invariance of weak topology with respect to linear transformations allows to extend these results on the problems of signal detection in Gaussian white noise and hypotheses testing on a solution of linear ill-posed problem with Gaussian noise.
 
 For hypothesis testing,
on a solution of linear ill-posed problem with Gaussian  noise,
 on a solution of deconvolution problem
 and for a signal detection in Gaussian white noise the necessary and sufficient conditions of existence of consistent and uniformly consistent tests are also provided in terms of weak topology (section 5).  Here the weak topologies in Hilbert space and in $L_2$ are implemented. The  conditions are akin to Le Cam and Schwartz Theorem.  The  results in terms of weak topology  seem rather unexpected. In estimation the conditions of consistency and uniform consistency are provided in terms of strong topologies. Thus one of the goal of the paper is to show  the role of weak topology in the study of consistency of tests.  Such an approach to the study of existence of  uniformly consistent and discernible tests has been developed Le Cam and Schwarts \cite{les}, Dembo and Peres \cite{dem} and Kulkarni and Zeitouni \cite{ku}. Note that the conditions of existence of uniformly consistent tests for signal detection in Gaussian white noise have been known earlier (Ermakov \cite{er00}). These conditions were given in another form (see Theorem \ref{t1}).

 For a long time it seems rather difficult to obtain  indistinquishability  conditions. Special examples of indistinquishability of hypotheses have been proposed. The first result on indistinguishability  of hypotheses in functional spaces has been obtained Le Cam \cite{le73}. For  hypothesis testing
on a density Le Cam \cite{le73} has established that the center of the ball and the interior of the ball in $L_1$ are indistinguishable. The further results were mostly related to the problem of signal detection in Gaussian white noise. Ibragimov and Hasminski \cite{ih} have shown  that the center of a ball and the interior of a ball in $L_2$ are indistinguishable. Burnashev \cite{bu} has proved the indistinguishability of
the interior of a ball in $L_p$- spaces, $p> 0$. For the problem of testing of simple hypothesis Janssen \cite{ja} has  showed that  any test can achieve high asymptotic power only on at most finite dimensional space of alternatives.
  For other setups we mention Ingster \cite{in93}
and Ingster and Kutoyants \cite{ik} works.  Ingster \cite{in93}
and Ingster and Kutoyants \cite{ik} showed  that we could not distinguish the center of a ball and the interior of a ball in $L_2$  in the problems of  hypothesis
testing on a density  and on an intensity function of Poisson process  respectively.

 The paper is organized as follows.
 The general setup and the definitions of consistency, pointwise consistency, distinguishability and discernibility are provided in section 2. Section 2 contains also the basic technique implemented to the proof of results. The abovementioned links of  conditions for existence of uniformly consistent, consistent, pointwise consistent and discernible tests  are established in section 3.   The necessary conditions and sufficient conditions for existence of
 consistent, uniformly consistent and discernible  tests for  hypothesis testing on a probability measure of i.i.d.r.v.'s are provided in section 4.
 Other problems of hypotheses testing are explored in section 5. If the proof of Theorem or Lemma does not provided in section, this proof  can be found   to the Appendix.

Denote by letters $c$ and $C$ generic constants. Denote $1_A(x)$ the indicator of  set $A$. Denote $[a]$ the integer part of $a \in R^1$. For any measures $P_1, P_2$ denote $P_1\otimes P_2$ the product of measures $P_1$ and $P_2$. For any measures $P_1, P_2$ we shall write $P_1 << P_2$ if measure $P_1$ is absolutely continuous with respect to $P_2$.
\section{Preliminaries \label{s2}}
\subsection{General setup. Definitions of consistency, pointwise consistency, discernibility, uniform consistency \label{ss21}}
  Let $  \frak{E}_n = (\Omega_n,  \frak{B}_n,  \frak{P}_n)$ be a sequence of statistical experiments  where $(\Omega_n,  \frak{B}_n)$  are sample spaces with $\sigma$-fields of Borel sets $\frak{B}_n$ and  $ \frak{P}_n= \{P_{\theta,n}, \theta \in \Theta\}$  are  families of probability measures.

We wish to test a hypothesis $H_0: \theta \in \Theta_0 \subset \Theta$ versus alternative $H_1: \theta \in \Theta_1 \subset \Theta$.

For any test $K_n$ denote $\alpha_{\theta}(K_n), \theta \in \Theta_0$, and $\beta_{\theta}(K_n), \theta \in \Theta_1,$ its type I and type II error probabilities respectively.

Denote
$$
\alpha(K_n) = \sup_{\theta \in \Theta_0} \alpha_{\theta}(K_n) \quad
\mbox{\rm and}
\quad
\beta(K_n) = \sup_{\theta \in \Theta_1} \beta_{\theta}(K_n).
$$
Below the definitions of different types of consistency are provided.

A sequence of tests $K_n$ is  pointwise consistent  (Lehmann and Romano \cite{le})  if
$$
\lim_{n\to \infty}\alpha_{\theta_0}(K_n)  = 0\quad
\mbox{\rm and}
\quad \lim_{n\to \infty}\beta_{\theta_1}(K_\epsilon)  = 0
$$
for all $\theta_0\in \Theta_0$ and $\theta_1 \in \Theta_1$.

A sequence of tests $K_n$ is  consistent (van der Vaart \cite{van}) if
$$
\lim_{n\to \infty} \alpha(K_n) = 0\quad
\mbox{\rm and}
\quad \lim_{n\to \infty}\beta_{\theta_1}(K_n)  = 0
$$
for all  $\theta_1 \in \Theta_1$.

A sequence of tests $K_n$  is uniformly consistent ( Hoefding and Wolfowitz \cite{ho}) if
$$
\lim_{n\to \infty} \alpha(K_n) = 0\quad
\mbox{\rm and}\quad
\lim_{n\to \infty} \beta(K_n) = 0.
$$
Hypothesis $H_0$ and  alternative $H_1$ are called distinguishable (Hoefding and Wolfowitz \cite{ho}) if  there is  uniformly  consistent sequence of tests.

Hypotheses $H_0$ and  alternative $H_1$ are called indistinguishable ( Hoefding and Wolfowitz \cite{ho}) if, for each $n$ and for each  test $K_n$, we have
$$\alpha(K_n) + \beta(K_n) \ge 1.$$

In subsections  5.1 and 5.2 we consider families of statistical experiments $  \frak{E}_\epsilon = (\Omega_\epsilon,  \frak{B}_\epsilon,  \frak{P}_\epsilon)$ depending on a continuous parameter $\epsilon>0$. For these setups we implement similar definitions of consistency.

 A sequence of tests $K_n$ is called discernible ( Devroye, Lugosi \cite{de} and Dembo, Peres \cite{dep}) or strong consistent (  van der Vaart \cite{van})  if
 \begin{equation}\label{m5}
 P(K_n = 1 \quad\mbox{for only finitely many}\quad n) =1 \quad\mbox{for all}\quad P \in \Theta_0
 \end{equation}
  and
 \begin{equation}\label{m6}
 P(K_n = 0\quad \mbox{for only finitely many}\quad n) =1 \quad\mbox{for all}\quad P \in \Theta_1.
 \end{equation}
 \subsection{Basic technique \label{ss23}} The proof of distinguishability is based on the following reasoning.

 Let we wish to test the hypotheses on a distribution of random variable $X$ defined on probability space $(\Omega, \frak{B},P)$. For the test $K(X)\equiv\alpha, 0 < \alpha <1$, we get
$$\alpha_P(K) + \beta_Q(K)=1\quad \mbox{for all}\quad P \in \Theta_0 \quad\mbox{and}\quad Q \in \Theta_1.$$
Thus, it is of interest,  to search for the test $K$ such that
$$
 \alpha_P(K) + \beta_Q(K) < 1-\delta,\quad \delta>0 \quad \mbox{for all}\quad P \in \Theta_0 \quad\mbox{and}\quad Q \in \Theta_1$$
 or, other words,
 \begin{equation}\label{e1}
 \int K dP + \int (1-K)\, dQ < 1-\delta.
 \end{equation}
 In this case we  say that  hypotheses and alternatives are weakly distinguishable.

For any $\epsilon > 0$ we can approximate the function $K$ by simple function
$$
K_0(x) = \sum_{i=1}^k c_i 1_{A_i}(x), \quad x \in \Omega
$$ such that
\begin{equation}\label{111} |K(x) - K_0(x)| < \epsilon, \quad x \in \Omega.
\end{equation}
Here $\{A_1,\ldots,A_k\}$ is a partition of $\Omega$.

Substituting $K_0$ in (\ref{e1}) and using (\ref{111}), we get
\begin{equation}\label{e2}
\sum_{i=1}^k c_i (Q(A_i)-P(A_i) )  \ge \delta - 2\epsilon.
\end{equation}
Hence we obtain the following Proposition.
\begin{proposition}\label{pr1} Hypothesis $H_0$ and alternative $H_1$ are weakly distinguishable  iff there is  a partition $A_1,\ldots,A_k$ of $\Omega$ such that the sets
$$
V_0 = \{v  = (v_1,\ldots,v_k) : v_1 = P(A_1),\ldots,v_k = P(A_k),\,\,\, P\in \Theta_0\} \subset R^k
$$
and
$$
V_1 = \{v  = (v_1,\ldots,v_k) : v_1 = Q(A_1),\ldots,v_k = Q(A_k),\,\,\, Q\in \Theta_1\} \subset R^k
$$
have disjoint closures.
\end{proposition}
Suppose we wish to test a hypothesis on probability measure of i.i.d.r.v.'s.
By Proposition \ref{pr1}, if the hypothesis and the alternative are weakly distinguishable, the problem can be reduced to hypothesis testing for the multinomial distribution.
 For this problem we immediately get the distinguishability. Thus the weak distinguishability implies  the distinguishability  (Hoefding and Wolfowitz \cite{ho}).

 For the problem of hypothesis testing on multinomial distribution the likelihood ratio tests have exponential decay of type I and type II error probabilities. Hence the  distinguishability implies also ( Le Cam \cite{le73} and Schwartz \cite{sch}) the existence of  sequence of tests $K_n$
 and constant $n_0$ such that
 \begin{equation}\label{x1}
 \alpha(K_n) \le \exp \{-cn\} \quad \mbox{and} \quad \beta(K_n) \le \exp\{-cn\}
\end{equation}
 for all $n> n_0$.

   The exponential decay of type I and type II error probabilities has been studied
  in a large number of papers (see Hoefding and Wolfowitz \cite{ho},  Dembo and Zeitouni \cite{dem}, Barron \cite{ba}, Ermakov \cite{er93} and references therein).

  Remark. In a more  cumbersome form similar distinquishability conditions have been proved Berger \cite{be} for the case of i.i.d.r.v.'s.  Le Cam  implemented the approximation of tests by simple functions   in the proof of Lemma 4 in \cite{le73} as auxiliary technique.
\section{Links of consistency, pointwise consistency, uniform consistency and discernibility}
 The results  are provided for  hypothesis testing on a probability measure of i.i.d.r.v.'s. For other setups (see subsection 5.1) similar results are obtained by easy modification of the reasoning.

Let $X_1,\ldots,X_n$ be i.i.d.r.v.'s on a probability space $(\Omega, \frak{B}, P)$
 where $\frak{B}$ is $\sigma$-field of Borel sets on  topological space $\Omega$.
 Denote $\Lambda$ the set of all probability measures on $(\Omega, \frak{B})$ and let $\Theta=\Lambda$.
 \begin{thm} \label{b2} There is a consistent sequence of tests  iff there are nested subsets   $\Theta_{1i} \subseteq \Theta_{1,i+1}, \Theta_1 = \cup_{i=1}^\infty \Theta_{1i},$ such that,
for each $i$,  the hypothesis $H_0: P \in\Theta_{0}$ and the alternative $H_{1i}: P \in \Theta_{1i}$ are distinguishable.
\end{thm}

The problems of distinguishability of  hypotheses with approaching sets of alternatives (see Ingster and Suslina \cite{is}, Comminges and  Dalalyan \cite{dal} and references therein) admits the interpretations in the famework of Theorem \ref{b2}. In these problems the sequence of distances assigning the approaching of alternatives allows to define consistent test. At the same time each alternative satisfies the distinguishability conditions.

\begin{thm} \label{b1} There is  a pointwise consistent  sequence of tests  iff there are nested subsets  $\Theta_{0i} \subseteq \Theta_{0,i+1}, \Theta_0 = \cup_{i=1}^\infty \Theta_{0i}$ and  $\Theta_{1i} \subseteq \Theta_{1,i+1}, \Theta_1 = \cup_{i=1}^\infty \Theta_{1i}$ such that,
for each $i$,  the hypothesis $H_{0i}: P \in\Theta_{0i}$ and the alternative $H_{1i}: P \in \Theta_{1i}$ are distinguishable.
\end{thm}
\begin{proof}[ Proof of Theorem \ref{b2}] Let $K_i, 1 \le i < \infty,$ be a consistent  sequence of tests. Let $0  < \alpha,\beta < 1$ be such that $ \alpha + \beta < 1$. For each $i$ define the subsets
$\Theta'_{1i} = \{P: \alpha(K_i) < \alpha, \beta_P(K_i) \le \beta,  P \in \Theta_1\}$. The sets $\Theta_0$ and $\Theta'_{1i}$ are weakly distinguishable  and therefore they are distinguishable.
 Hence the hypothesis $H_0$ and the alternative $H_{1i}: P \in \Theta_{1i} = \cup_{j=1}^i \Theta'_{1j}$ are  distinguishable by Lemma \ref{l33} given below.

\begin{lemma}\label{l33} Let a hypothesis $H_0: P \in \Theta_0$ be distinguishable for alternatives
$H_{11}: P \in \Theta_{11}$ and  $H_{12}: P \in \Theta_{12}$. Then the hypothesis $H_0: P \in \Theta_0$ is  distinguishable for the alternative
$H_{1}: P \in \Theta_{11}\cup\Theta_{12}$.
\end{lemma}
\end{proof}

The proof of Theorem \ref{b1} is similar.

  \begin{thm}\label{b3} If there is a pointwise consistent sequence of tests, then there is a discernible  sequence of tests.
\end{thm}
The exponential decay of type I and type II error probabilities (see (\ref{x1})) allows to prove uniform discernibility of distinguishable sets of hypotheses and alternatives.

We say that  a sequence of tests $K_n$ is uniformly discernible if
\begin{equation}\label{eq12}
\lim_{k\to\infty}\sup_{P \in \Theta_0} P(K_n =1 \,\,\,\mbox{for some}\,\,\, n> k) = 0
\end{equation}
and
\begin{equation*}
\lim_{k\to\infty}\sup_{P \in \Theta_1} P(K_n =0\,\,\, \mbox{for some}\,\,\, n> k) = 0.
\end{equation*}
If there is uniformly discernible  sequence of tests, we  say that the sets of hypotheses and alternatives are uniformly discernible.
\begin{thm}\label{b4} If  hypothesis $H_0$ and  alternative $H_1$ are distinguishable, then they are uniformly discernible.

For a sequence of tests $K_n$ satisfying (\ref{x1}), there are positive constants $c$ and $C$ such that
\begin{equation}\label{x2}
\sup_{P \in \Theta_0} P(K_n =1\,\,\, \mbox{for some}\,\,\, n> k) \le C\exp\{-ck\}
\end{equation}
and
\begin{equation}\label{x3}
\sup_{P \in \Theta_1} P(K_n =0\,\,\, \mbox{for some}\,\,\, n> k) \le C\exp\{-ck\}.
\end{equation}
\end{thm}
\begin{thm}\label{b4a} If there is  a consistent sequence of tests,  then there is a discernible sequence of tests satisfying (\ref{eq12}). \end{thm}

 We say that a sequence of tests is strongly consistent if this sequence of tests is discernible and (\ref{eq12}) holds.

\noindent{\sl Remark.} Pfanzagl (Theorem 2.1, \cite{pf}) has established that, if there is uniformly consistent estimator, then there is strongly uniformly consistent estimator. This implies that the existence of uniformly discernible tests follows from the existence of uniformly consistent tests. In Theorem \ref{b4} additional estimators of rates of convergence are provided.
Pfanzagl (Theorem 2.2, \cite{pf}) has proved also that, if there is consistent estimator then there is strongly consistent estimator a.e. for any a priori Bayes measure on the set of parameters. Hence similar statement for the problem of existence of discernible tests follows. Theorem \ref{b3} contains the stronger result.
\section{\bf  Hypothesis testing on a probability  measure of independent sample.\label{s3}}
Different approaches are implemented to the study of existence of consistent, uniformly consistent and discernible tests, among them,

implementation of estimators as test statistics,

distance method,

convergence of probability measures in weak topologies generated continuous or measurable functions.

As mentioned the implementation of estimators as test statistics will be treated in the frameworks of distance method.
\subsection{Weak topologies}
Let $\Psi$ be a set of measurable functions $f:\Omega \to R^1$. The coarsest topology in $\Lambda$ providing the continuous mapping
$$
P \to \int_\Omega f \, dP, \quad P \in \Lambda
$$
for all $f \in \Psi$ is called the $\tau_\Psi$-topology of weak convergence.

If $\Psi$ is the set of all bounded continuous functions, the $\tau_\Psi$ - topology is   the weak topology.
If $\Psi$ is the set of  indicator functions of all measurable Borel sets, the $\tau_\Psi$ - topology is called the $\tau$-topology (see Groeneboom,  Oosterhoff and Ruymgaart \cite{gor}) or the topology of setwise convergence on all Borel sets
  (see Ganssler \cite{ga} and Bogachev \cite{bo}). In the papers the $\tau$-topology is often replaced with the $\tau_\Phi$-topology with the set $\Phi$ of all bounded measurable functions. All results given below coincide for these topologies.

 For any set $A \subset \Lambda$ denote $\frak{cl}_{\tau_w}(A)$, $\frak{cl}_{\tau}(A)$  and $\frak{cl}_{\tau_\Phi}(A)$ the closures of $A$ in weak, $\tau$ and $\tau_\Phi$-topologies respectively. We shall write $\tau_\Psi$ if the weak topology may be chosen arbitrary: weak, $\tau$ or $\tau_\Phi$.

 Theorem \ref{b7} given below is a version of Le Cam and Schwartz Theorem (see \cite{les}, p.141).

 \begin{thm} \label{b7} {\sl i.} Let the set $\Theta_0$ be relatively compact in the $\tau_\Psi$ - topology. Then the hypothesis $H_0$ and the alternative $H_1$ are distinguishable if the closures of sets $\Theta_0$ and $\Theta_1$ are disjoint in the $\tau_\Psi$-topology.

{\sl ii.}  If $\Theta_0$  and $\Theta_1$ are  relatively compact in the $\tau$ - topology, then the converse holds if the closures of $\Theta_0$ and $\Theta_1$ in the $\tau$-topology are disjoint.
\end{thm}

{\sl Example} 2.1. Let $\nu$ be   Lebesgue measure in $(0,1)$ and let we wish to test a hypothesis $H_0 : f(x) =1,\, x \in (0,1)$  on a density $f = \frac{dP}{d\nu}$. The alternative is $H_1: f  =f_i(x) = 1 + \sin(2\pi i x),  1 \le i < \infty$.

 For any measurable set $B \in \frak{B}$ we have
  $$
 \lim_{i\to\infty} \int_B f_i(x) \, dx = \int_B dx.
 $$
 Therefore the hypothesis $H_0$ and the alternative $H_1$ are indistinguishable.

\begin{proof}[ Proof of   Theorem \ref{b7} {\sl ii}]  The partition $\{A_1,\ldots,A_k\}$ satisfying Proposition \ref{pr1} may exist only for the set $\Omega^d$ with $d>1$. This is the main problem in the proof of distinguishability. If $\Theta$ is relatively compact, the map $\Theta \to \Theta \otimes \Theta$ is continuous in  the $\tau_\Phi$- topology (Proposition 1, Le Cam and Schwartz \cite{les} ). Therefore the existence of the partition for $\Omega^d$ implies the existence of such a partition for $\Omega$.\end{proof}
The proof of  Theorem \ref{b7} {\sl i.} is akin to the proof of  Theorem 2 {\sl i.} in Dembo and Peres \cite{dep}. For the $\tau$-topology the test statistics are easily defined on the base of relative frequencies  of events $A_1,\ldots,A_k$. For other topologies the indicator functions are replaced with  functions $\phi_1,\ldots,\phi_k \in \Psi$.

  Le Cam and Schwartz \cite{les}, p.141, provided the necessary and sufficient conditions of distinguishability in the following form.
\begin{thm}\label{b8} Let sets $A \subset \Lambda$  and $B \subset \Lambda$ be relatively compact in the $\tau_\Phi$-topology. Then the set $A$ of hypotheses and the set $B$ of alternatives are distinguishable iff "there is a finite family $\{f_j, j=1,\ldots,m\}$ of measurable bounded functions on" $\Omega$ "such that
\begin{equation}\label{x11}
\sup\left| \int f_j \, dP - \int f_j dQ\right| < 1
\end{equation}
implies that either both $P$ and $Q$ are elements of $A$ or both are elements of $B$".
\end{thm}
If we define functions $f_j, 1 \le j \le m=k$ as  the indicator functions of the sets of the partition $\{A_1,\ldots,A_k\}$  multiplied on some constants, we get that  (\ref{x11}) holds if the closures of $A$ and $B$ are disjoint.

If $\Omega$ is a metric space,  then the weak topology and the $\tau$ - topology coincide on compacts in the $\tau$-topology ( Ganssler \cite{ga}, Lemma 2.3).

{\sl Remark}. Hoefding and  Wolfowitz \cite{ho} have obtained the necessary and sufficient conditions of distinquishability in another form. The necessary distinguishability conditions  were proposed in terms of variational metric on the set of probability measures. The probability measures are supposed to satisfy  some assumptions  on  finiteness of number of intervals of monotonicity of differences of distribution functions of hypotheses and alternatives.

The  results on consistency and pointwise consistency will be provided in terms of $F_\sigma$-sets and  $\sigma$-compact sets.
Set $\Psi$ is called $F_\sigma$-set if $\Psi$ is countable union of closed sets.
Set  $\Psi$ is called $\sigma$-compact   if $\Psi$ is countable union of compacts.

{\it Example}. Let probability measures of  set $\Psi \subset \Lambda$ be absolutely continuous for probability measure $\mu$. Suppose that, for each measure $P \in \Psi$,  there is some $p>1$ such that $dP/d\mu \in L_p(d\mu)$. Then $\Psi$ is $\sigma$-compact set in the $\tau$-topology.

The other examples of $\sigma$-compact sets can be provided using Orlich spaces.

Implementing Theorems \ref{b2} - \ref{b3}, we can easily derive from Theorem \ref{b7}   necessary and sufficient  conditions for existence of consistent and pointwise consistent tests.
\begin{thm}\label{b9} \sl{i.} Let $\Theta_0$ be relatively compact in the $\tau_\Psi$-topology. Then there are consistent  tests if $\Theta_0$ and $\Theta_1$  are contained respectively in disjoint closed set and $F_\sigma$-set in the $\tau_\Psi$-topology.

{\sl ii.} For the $\tau$-topology, the converse holds if we suppose additionally that $\Theta_1$ is contained in  some $\sigma$-compact set.
\end{thm}

\begin{thm}\label{b9a} \sl{i.} There are pointwise consistent  tests if $\Theta_0$ and $\Theta_1$  are contained respectively in disjoint $\sigma$-compact set and $F_\sigma$- set in the $\tau_\Psi$-topology.

{\sl ii.} For the $\tau$-topology, the converse holds if we suppose additionally that  $\Theta_1$ is contained in  some $\sigma$-compact set.
\end{thm}

A sequence of probability measures $P_m$  converging to some probability measure $P$ is compact in the $\tau$-topology ( Ganssler (1971)). Thus the condition of indistinguishability
   can be given in the following form.
\vskip 0.5cm
   For any set $\Upsilon \in \Lambda$, denote $\mbox{cl}_{s\tau}(\Upsilon)$ the sequential closure of $\Upsilon$ in the $\tau$-topology.
  \begin{thm}\label{seq1} For any sets $\Theta_0, \Theta_1 \subset \Lambda$,  $$\mbox{cl}_{s\tau}(\Theta_0) \cap \mbox{cl}_{s\tau} (\Theta_1) \ne \emptyset$$ implies
    indistinguishability of $H_0$ and $H_1$.
    \end{thm}
     {\bf Remark.} We could not prove the indistinguishability of any sets $\Theta_0, \Theta_1 \subset \Lambda$
   such  that $\mbox{cl}_\tau(\Theta_0) \cap \mbox{cl}_\tau (\Theta_1) \ne \emptyset$.
   As showed Peres, the map $P \to P\otimes P$ of $\Lambda \to \Lambda\otimes\Lambda$ is not continuous in the $\tau$-topology (8.10.116 in Bogachev \cite{bo}).

\subsection{ Necessary conditions of distinguishability and distance on variation}
   The distance on variation is a standard tool for the study of distinguishability of hypotheses (Hoefding and Wolfowitz \cite{ho}, Le Cam \cite{le73}, Lehmann and Romano \cite{le},
van der Vaart \cite{van}, Devroye and Lugosi \cite{de}, Ingster and Suslina \cite{is} and other numerous papers).

   For any probability measures $P<<\nu$ and $Q<<\nu$ define the variational distance
$$
\mbox{\rm var}(P,Q) = \frac{1}{2}\int_\Omega |dP/d\nu - dQ/d\nu| \,d\nu.
$$
For any sets of probability measures $A$ and $B$ denote
$$
\mbox{\rm var}(A,B) = \inf\{\mbox{var}(P,Q):\, P\in A, Q \in B\}.
$$
Denote $[A]$  the convex hull of set $A \subset\Lambda$.

The proof of distinguishability is based usually on the following Theorem (Kraft \cite{kr}).
\begin{thm}\label{t5} Let  probability measures in $\Theta_0 \cup \Theta_1$ be absolutely continuous with respect to measure $\nu$. Then, for any test $K$, there holds
\begin{equation}\label{m2}
\alpha(K,\Theta_0) + \beta(K,\Theta_1) \ge 1 - \mbox{\rm var}([\Theta_0],[\Theta_1])
\end{equation}
and this lower bound is attained.
\end{thm}
A sequence of densities $f_k$ converges to density $f_0$ in $L_1(\nu)$  (see Dunford and Schwarz \cite{du}, Th.12, sec.8, Ch IV),
if, for any measurable set $B \in \frak{B}$, there holds
 \begin{equation}\label{m3}
 \lim_{k\to\infty} \int_B f_k\, d\nu= \int_B f_0 d\nu
\end{equation}
 and $f_k$ converges to $f_0$ in measure.

 If (\ref{m3}) holds and sequence $f_k$ does not converges to $f_0$ in measure, we could not distinguish
 the set of hypotheses $\Theta_0=\{f_0\}$ and the set of alternatives $\{f_1,f_2,\ldots\}$.
 By Mazur Theorem (see Yosida \cite{yo}, Th.2, sec.1, Ch.5 ), the weak convergence $f_k$ to $f_0$ implies the convergence of convex combinations of $f_k$ to $f_0$ in $L_1(\nu)$. Therefore the right-hand side of (\ref{m2}) equals one.
\subsection{Uniform consistency and distance method}    Let $  \frak{E}_n = (\Omega_n,  \frak{B}_n,  \frak{P}_n)$ be a sequence of statistical experiments  with families of probability measures  $ \frak{P}_n= \{P_{\theta,n}, \theta \in \Theta\}$. Suppose   the set $\Theta$ is a metric space with a metric $\rho$.
Let $X_n$ be random variable on probability space $(\Omega_n,  \frak{B}_n,  P_{\theta,n})$ and let $\hat\theta_n = \hat\theta_n(X_n)$ be an estimator of parameter $\theta$.

  Hoefding and Wolfowitz \cite{ho} proposed  classification of metrics on  consistent and uniformly consistent.
The metric $\rho$ is consistent in $\tilde\Theta\subset\Theta$, if, for each $\epsilon >0$  and for each $\theta \in \tilde\Theta$, there holds
\begin{equation}\label{x6}
\lim_{n\to\infty} P_\theta(\rho(\hat \theta_n,\theta) > \epsilon) = 0.
\end{equation}
The distance $\rho$ is uniformly consistent in $\tilde\Theta$ if the convergence in (\ref{x6}) is uniform for $\theta \in \tilde\Theta$.
The consistency (uniform consistency) of metric is equivalent the consistency (uniform consistency) of estimator.

Theorem \ref{b21} is a version of Theorem proved Hoefding and Wolfowitz  (Th. 3.1, \cite{ho}).

\begin{thm}\label{b21} Let $\rho$ be uniformly consistent in $\Theta_0 \cup \Theta_1$. Then the hypothesis and the alternative are distinguishable if
\begin{equation}\label{x7}
\rho(\Theta_0,\Theta_1) = \inf\{\rho(\tau,\eta) : \tau \in \Theta_0, \eta \in \Theta_1\} >0.
\end{equation}
\end{thm}
As wellknown the Kolmogorov-Smirnov distance and  the distance corresponding omega-squared test are uniformly consistent.

Theorem \ref{b21} provides sufficient conditions for distinquishability of hypotheses. Below we show that these conditions are necessary if some additional assumptions hold.
 We narrow the setup and consider the problem of hypothesis testing on a value of functional $T:\Lambda \to \Theta$ on a sample of i.i.d.r.v.'s.

   We say that the set $V \subset \Psi$ is saturated, if for any $\eta \in V$ and any sequence  $\eta_n \in V,$ $ \eta_n \to \eta $ as $n \to \infty $,  there are $P \in \Lambda$ and a sequence   $P_n \in \Lambda$  such that $T(P) = \eta, T(P_n) = \eta_n$ and $P_n \to P$ as $n \to \infty$ in  the $\tau$-topology.

   If the sets $\Theta_0$ and $\Theta_1$ are saturated, we can implement the indistinguishability criterion of Theorem \ref{seq1}.

      \begin{thm}\label{5a}  If $\Theta_0$ is relatively compact and $\Theta_0$ and $\Theta_1$ are saturated, then  (\ref{x7}) is necessary.
   \end{thm}

   {\it Example. Kolmogorov tests.} Let $\Omega = (0,1)$ and let $\rho(Q,P) = \max_{x \in (0,1)} |F(x) -   F_0(x)|$ where $F(x)$ and $F_0(x)$ are distribution functions of probability measures $P$ and $Q$. Denote $P_0$ the Lebesque measure. Define the probability measures $P_u = P_0 + uG, 0\le u <1,$ where  the signed measure $G$ has the density $dG/dP_0(x) = -1$ if $x \in (0,1/2]$ and $dG/dP_0(x) = 1$ if $x \in (1/2,1)$.
Then, for any $u, 0\le u < 1$ and $u_n \to u$ as $n \to \infty$ one can put  $P = P_u$ and $P_n = P_{u_n}$ in the definition of saturated set.
\vskip 0.25cm
Denote $\Lambda_0$ the set of all signed measures $G$ having bounded variation and such that $G(\Omega)=0$.
The definition of the $\tau$-topology in $\Lambda_0$ is akin to the definition in $\Lambda$.

Suppose the functional $T:\Lambda \to R^d$ satisfies the following  condition of Hadamard  differentiability.
\vskip 0.25cm
\noindent{\bf A.} For any $P\in\Lambda$ with $T(P) \in \Theta_0 \cup \Theta_1$, there is vector function  $\vec h = (h_1,\ldots,h_d), h_i \in L_1(P), E[h_i(X)] = 0, 1 \le i \le k$, such that the functions $h_1,\ldots,h_d$ are linear independent, $E_P[h_i(X)]  = 0, 1\le i <d$,  and for any $G \in \Lambda_0, G << P$  we have
\begin{equation*}
\left| T(P + t_n G_n) - T(P) - t_n\int \vec h d G \right| = o(|t_n|)
\end{equation*}
as $n \to \infty$, for all converging sequences $t_n \to 0$ and for all sequences $G_n \to_{\tau} G$ with $G_n \in \Lambda_0, P + t_n G_n \in \Lambda$.
\vskip 0.25cm
In A we do not introduce the metric in $\Lambda$.  In the definition of Hadamard derivative  the existence of metric in $\Lambda$  is required usually in statistical problem. If this metric is continuous in the $\tau$-topology, Theorem \ref{b6w} hold for such a definition of Hadamard differentiability as well. Note that, in A, it suffices to consider only the signed measures $G$ with the densities $\frac{dG}{dP} = \sum_{i=1}^d c_i h_i, c_i \in R^1.$

\begin{thm}\label{b6w} Assume A. Then the functional $T$ is saturated.
\end{thm}

{\sl Remark.} Theorems \ref{b21} and \ref{5a} together with Theorems \ref{b2} and \ref{b1} allow to point out the conditions for existence of consistent and pointwise consistent tests.  In these setups  the distances can be different for each pairs of sets $\Theta_0, \Theta_{1i}$ and $\Theta_{0i} , \Theta_{1i}$ respectively in the conditions for existence of consistent and pointwise consistent tests. Here the sets $\Theta_0, \Theta_{1i}$ and $\Theta_{0i} , \Theta_{1i}$ are the same as in Theorems \ref{b2} and \ref{b1}. Since these statements are rather evident and are rather cumbersome we omit them. For  discernibility of hypotheses similar necessary and sufficient conditions were pointed out Dembo and Peres (Theorem 1 i., \cite{de}) for hypothesis testing on a sample mean. These are the only necessary conditions obtained for the setup of this subsection.

There are natural topological  extensions of Theorem \ref{b21} on the problems of consistent and pointwise consistent hypotheses testing. If the estimator is uniformly consistent on the closed set of hypotheses and consistent on the open set of alternatives and these sets are disjoint, then there are consistent tests. Pfanzagl (Lemma 8.1, \cite{pf}) has proved similar statement for strongly consistent estimators and strongly consistent tests. This statement is obtained by easy modification of Pfanzagl reasoning. If the disjoint sets of hypotheses and alternatives are open and there is consistent estimator on these sets, then there are pointwise consistent tests. A version of this statement for discernible tests one can find in Dembo and Peres (Theorem 1 ii., \cite{de}).

\section{Other problems of hypotheses testing}
 
\subsection{Hypothesis testing on Gaussian measure}

 Let $X_1,\ldots,X_n$ be i.i.d. Gaussian random vectors in separable Hilbert space $H$. Denote $S = E X_1$ and let $R$ be covariance operator of $X_1$.

 The problem is to test the hypothesis $S \in \Theta_0 \subset H$ versus alternative $S \in \Theta_1 \subset H$.

 \begin{thm}\label{tg1} Let the sets $\Theta_0$  and $\Theta_1$ be relatively compact. Then the hypothesis $H_0$ and alternative $H_1$ are distinguishable, iff the closures of $\Theta_0$  and $\Theta_1$ are disjoint. \end{thm}

 Proof. If the set of probability measures is tight, then the $\tau$-topology and weak topology coincide. Implementing Le Cam -Schwartz Theorem and Theorem 3.7.7, Bogachev \cite{bog} on weak convergence of probability measures we get Theorem  \ref{tg1}.

 Let $\Theta$ be tight set of centered Gaussian measures  (on tightness criteria see Theorem 3.7.10 in Bogachev \cite{bog}). Denote $\Psi$ - the set of their covariance operators. The problem is to test the hypothesis $H_0: R \in \Psi_0 \subset \Psi$ versus alternative $H_1: R \in \Psi_1 \subset \Psi$.

For $i=1,2$, define the sets $\Upsilon_i = \{ R^{1/2}:
 R \in \Psi_i\}.$

 \begin{thm}\label{tg2} The hypothesis $H_0$ and alternative $H_1$ are distinguishable, iff the closures of $\Upsilon_0$  and $\Upsilon_1$ in Hilbert - Schmidt norm are disjoint.  \end{thm}

 Proof of Theorem \ref{tg2} differs from the proof of Theorem \ref{tg1}  only implementation of another criteria of weak convergence of centered Gaussian measures   (see Th 3.7.10, Bogachev \cite{bog}).

\subsection{Signal detection  in $L_2$ \label{s2.2}}

 Suppose we  observe a realization of stochastic process $Y_\epsilon(t), t \in (0,1)$, defined by the stochastic differential equation
$$
dY_\epsilon(t) = S(t) dt + \epsilon dw(t), \quad \epsilon > 0
$$
where $S\in L_2(0,1)$ is unknown signal and $dw(t)$ is Gaussian white noise.

We wish to test a hypothesis $H_0: S\in\Theta_0 \subset L_2(0,1)$ versus alternative $H_1: S\in\Theta_1 \subset L_2(0,1)$.

The results are provided in terms of the weak topology in $L_2(0,1)$.

\begin{thm}\label{si1} {\sl i.} Let $\Theta_0$ be bounded set in $L_2$. Then $H_0$ and $H_1$ are distinguishable iff  the closures of $\Theta_0$ and $\Theta_1$ are disjoint.

{\sl ii}. The converse holds if $\Theta_1$ is also bounded sets in $L_2$.
\end{thm}

\begin{thm}\label{si121}
{\sl i.} Let $\Theta_0$  be bounded set in $L_2$. Then there are consistent tests iff  $\Theta_0$ and $\Theta_1$  are contained  in disjoint closed set and $F_\sigma$- set respectively.

{\sl ii.} There are point-wise consistent tests iff the sets $\Theta_0$ and $\Theta_1$ are contained in disjoint $F_\sigma$-sets.
\end{thm}

\begin{proof}[ Proof of Theorem \ref{si1} { \sl i}]
 By Tikhonov theorem, $\Theta_0$ is relatively compact. Therefore $\Theta_0$ and $\Theta_1$ have disjoint finite covering by the sets
$$U_l = \cap_{i=1}^{k_l} U_{il}, \quad U_{il} = U(f_{il},c_{il},S_{il}) =\{ S: \int f_{il}(s)(S(s)- S_{il}(s)) ds \le c_{il}, S \in L_2\},
$$
with $f_{il}, S_{il}\in L_2, 1 \le i \le k_l, 1 \le l \le m$. Thus we can define uniformly consistent tests using statistics $\int f_{il}(s)(dY_\epsilon(s)- S_{il}(s) ds )$.    \end{proof}

\begin{proof}[ Proof of Theorem \ref{si121}] Theorem \ref{si1}  implies that the
bounded sets $\Theta_0$ and $\Theta_1$ are distinguishable only if there are uniformly consistent tests depending on a finite number of linear statistics
$$
\int S_1(t) dY_\epsilon(t), \ldots, \int S_k(t) dY_\epsilon(t)
$$
with $S_1,\ldots,S_k \in L_2(0,1)$.

This allows to prove versions of Theorems \ref{b2} and \ref{b1} for discrete values of parameter $\epsilon = \epsilon_n =Cn^{-1/2}$ and to obtain  Theorem \ref{si121} for such values of parameter $\epsilon = \epsilon_n$. Hence the necessary conditions follow. The test statistics constructed for the parameters $\epsilon_n$ work for arbitrary $\epsilon > 0$.  This implies the sufficiency.\end{proof}
\subsection{Hypothesis testing on a solution of linear ill-posed problem \label{s2.4}}
In Hilbert space $H$
  we wish to test a hypothesis on a  vector $\theta \in \Theta \subset H$ from
  the observed  Gaussian random vector
 \begin{equation}\label{il1}
  Y = A\theta + \epsilon \xi, \quad \epsilon > 0.
  \end{equation}
  Hereafter $A: H \to H$ is known  linear operator and $\xi$ is Gaussian random vector having known covariance operator $R: H \to H$ and
  $E\xi = 0$.

For any operator $U: H \to H$ denote $\frak{R}(U)$ the rangespace of $U$.

  Suppose that the nullspaces of $A$ and $R$ equal  zero and $\frak{R}(A) \subseteq \frak{R}(R^{1/2})$.

\begin{thm}\label{t21} Let the operator $R^{-1/2}A$ be bounded. Then the statements  of Theorems \ref{si1} and \ref{si121} hold for the weak topology in $H$.\end{thm}

\begin{proof} The  statements   of Theorems \ref{si1} and \ref{si121} hold for the sets of hypotheses $R^{-1/2}A\Theta_0$ and alternatives $R^{-1/2}A\Theta_1$. Thus it suffices to implement the inverse map $A^{-1}R^{1/2}$ to (\ref{il1}).\end{proof}

The problem of  signal  detection in the heteroscedastic Gaussian white noise can be considered as a particular case of  Theorem \ref{t21}.

Let we observe a random process $Y(t), t \in (0,1),$ defined the stochastic differential equation
$$
dY(t) = S(t) dt + \epsilon h(t) dw(t), \quad \epsilon > 0
$$
where $S\in L_2(0,1)$ is unknown signal, $h(t)$ is a weight function.

One needs  to test a hypothesis on a signal $ S$.
\begin{thm}\label{t3} Let $0< c <  h(t) < C < \infty$ for all $t \in (0,1)$. Then the statements  of Theorems \ref{si1} and \ref{si121} hold for the weak topology in $L_2(0,1)$.\end{thm}
\subsection{Hypothesis testing on a solution of deconvolution problem \label{s2.5}}
Let we observe i.i.d.r.v.'s $Z_1,\ldots,Z_n$ having a density $h(z), z \in R^1,$ with respect to Lebesgue measure. It is known that $Z_i = X_i + Y_i, 1 \le i \le n,$ where $X_1,\ldots,X_n$ and $Y_1,\ldots,Y_n$ are i.i.d.r.v.'s with  densities $f(x), x \in R^1,$ and $g(y), y \in R^1,$ respectively.  The density $g$ is known.

 We wish to test a hypothesis $H_0: f \in \Theta_0$ versus  alternative $H_1: f \in \Theta_1$ where $\Theta_0, \Theta_1 \subset L_2(R^1)$.

Suppose
 $g \in L_2(R^1)$.

 Denote
 $$
\hat g(\omega) = \int_{-\infty}^{\infty} \exp\{i\omega x\} g(y) \, dy, \quad \omega \in R^1.
$$
Denote $\Psi_0$ and $\Psi_1$ the sets of probability measures of $\Theta_0$ and $\Theta_1$ respectively.
\begin{thm}\label{t23} Suppose the sets $\Psi_0$ and $\Psi_1$ are tight. Let
$$\mbox{\rm essinf}_{\omega \in (-a,a)}|\hat g(\omega)| \ne 0$$ for all $a > 0$. Then the statements of  Theorem \ref{si1} holds for the weak topology in $L_2(R^1)$.\end{thm}
\begin{proof} Note that the sets of probability measures having the densities from $L_2(R^1)$ are  equicontinuous. Therefore we can implement the same reasoning as in the proof of Theorem \ref{t21} and to reduce the problem to the setup of Theorem \ref{b7}.\end{proof}
\subsection{Hypothesis testing on a  mean measure of Poisson random process \label{s2.3}} Let we be given $n$ independent realizations $\kappa_1,\ldots,\kappa_n$ of Poisson random process with mean measure $P$  defined on Borel sets of $\sigma$-field $\frak{B}$.
The problem is to test a hypothesis $H_0: P \in \Theta_0\subset\Theta$ versus  alternative $H_1: P \in \Theta_1\subset\Theta$ where  $\Theta$ is the set of all measures $P, P(\Omega) < \infty$.

Denote $N_n$  the number of atoms of Poisson random process $\kappa_1+\ldots+\kappa_n$.

We have ((6.2) in Arcones \cite{ar})
\begin{equation}\label{ar1}
\begin{split}&
P(|N_n - n\lambda| > x) \le \exp\{-n(\lambda+x)\log(1+x/\lambda) + nx\}\\& + \exp\{-n(\lambda - x)\log(1-x/\lambda)- nx\}
\end{split}
\end{equation}
with $\lambda = P(\Omega)$.

The conditional distribution of  $P(\kappa_1+\ldots+\kappa_n|N_n =k)$ coincide with the distribution of empirical probability measure $\hat Q_k$ of i.i.d.r.v's $X_1,\ldots,X_k$ having the probability measure $Q(A) = P(A)/P(\Omega), A \in \frak{B}$.
This statement and inequality (\ref{ar1}) allow to extend the results of section 4 on the problem of hypothesis testing on a mean measure of Poisson process. Below a version of Theorem \ref{b7} only will be provided.
\begin{thm}\label{t4.1}  {\sl i.} Let $\Theta_0$ be relatively compact in the $\tau_\Psi$ - topology. Then the hypothesis $H_0$ and the alternative $H_1$ are distinguishable if the closures of $\Theta_0$ and $\Theta_1$ are disjoint.

{\sl ii.} If $\Theta_0$  and $\Theta_1$ are  relatively compact in the $\tau$ - topology, then the converse holds if the closures of $\Theta_0$ and $\Theta_1$ in the $\tau$-topology are disjoint.
 \end{thm}
 \begin{proof}[ Proof of   Theorem \ref{t4.1} {\sl ii}] By Theorem 2.6 in Ganssler \cite{ga}, the sets $\Theta_0$ and $\Theta_1$ are relatively sequentially compact.

Suppose that the sets $\Theta_0$ and $\Theta_1$ have common limit point $P$ and are not indistinguishable.  Then there exist sequences
$ P_k \in \Theta_0$ and $Q_k \in \Theta_1$ converging to $P \in \Theta$.

For any test $K_n$, for any $l$, we have
\begin{equation*}
\lim_{k\to\infty} E_{P_k}(K_n|N_n=l) = E_P(K_n|N_n=l),
\end{equation*}
\begin{equation*}
\lim_{k\to\infty} E_{Q_k}(1-K_n|N_n=l) = E_P(1-K_n|N_n=l)
\end{equation*}
and
\begin{equation*}
\lim_{k\to\infty}P_k(N_n = l) =   \lim_{k\to\infty} Q_k(N_n = l) = P(N_n =l).
\end{equation*}
Hence the sets $\Theta_0$ and $\Theta_1$ are indistinguishable and we have a contradiction.\end{proof}

The proof of {\sl i.} is akin to the proof of  Theorem \ref{b7} {\sl i.} and is omitted.
\section{Appendix}
 \begin{proof}[Proof of Lemma \ref{l33}] By Proposition  \ref{pr1}, there are $m_i, i=1,2,$ and measurable  partitions $A_{i1},\ldots,A_{ik_i}$ of $\Omega^{m_i}$ such that the sets
$$
V_{0i} = \{v  = (v_1,\ldots,v_k) : v_1 = P^{m_i}(A_{i1}),\ldots,v_k = P^{m_i}(A_{ik_i}), P\in \Theta_0\} \subset R^{m_i}
$$
and
$$
V_{1i} = \{v  = (v_1,\ldots,v_k) : v_1 = P^{m_i}(A_{i1}),\ldots,v_k = P^{m_i}(A_{ik_i}), P\in \Theta_{1i}\} \subset R^{m_i}
$$
have disjoint closures for each $i=1,2$. Hereafter $P^{m_i}$ denotes $m_i$-fold product of probability measure $P$.

Since there is uniformly consistent estimator of $P^{m_1}(A_{11}),\ldots,P^{m_1}(A_{1k_1})$,  $P^{m_2}(A_{21}),\ldots,P^{m_2}(A_{2k_2})$, then there are uniformly consistent  tests for  testing the hypothesis $H_0: P \in \Theta_0$ versus the alternative $H_1: P \in \Theta_{11}\cup \Theta_{12}$.\end{proof}

 \begin{proof}[Proof of Theorem \ref{b3}] Let $\Theta_{0i}$  and $\Theta_{1i}, i=1,2,\ldots$ be sequences of subsets of $\Theta_{0}$  and $\Theta_{1}$ such that there are uniformly consistent sequences of tests for these subsets. By (\ref{x1}), for each $i$, there are tests $K_{ni}$ such that, for $n > n_{0i}$, we have
\begin{equation}\label{r1}
 \alpha(K_{ni}) \le \exp \{-c_in\} \quad \mbox{and} \quad \beta(K_{ni}) \le \exp\{-c_in\}.
\end{equation}
By  Borel-Cantelli  Lemma, for the proof of Theorem \ref{b3}, it suffices to define  sequence of tests $K_n$ such  that
\begin{equation}\label{r2}
\sum_{n=1}^\infty \alpha(K_n,P) < \infty \quad \mbox{\rm for each } \quad P \in \Theta_0
\end{equation}
and
\begin{equation}\label{r3}
\sum_{n=1}^\infty \beta(K_n,Q) < \infty \quad \mbox{\rm for each } \quad Q \in \Theta_1.
\end{equation}
Define the numbers $l_i$ such that $l_1= 1$ and
$$\exp\{-c_il_i\}(1 - \exp\{-c_i\})^{-1}  \le i^{-2}.$$
Choose a sequence
$ l_{i_0}<l_{i_1} < l_{i_2} < l_{i_3}<\ldots$
with $i_0 =1$ and $l_{i_s} > n_{0i_s}, 1 \le s < \infty$.

Put
$$
K_1 = K_{11},\ldots,K_{l_{i_1}} = K_{1l_{i_1}},
 $$
 $$
 K_{l_{i_1}+1} = K_{i_1,l_{i_1}+1},\ldots,K_{l_{i_2}} = K_{i_2,l_{i_2}}, K_{l_{i_2}+1} = K_{i_2,l_{i_2}+1}, \ldots
$$
 Then, for any $P \in \Theta_{0i_t}, 1 \le t < \infty$, we have
\begin{equation*}
\sum_{n=l_{i_t}}^\infty \alpha(K_n,P) \le \sum_{s=t}^\infty\sum_{n =l_{i_s}+1}^{l_{i_{s+1}}} \alpha(K_{i_sn})  \le
 \sum_{s=t}^\infty i_s^{-2} < \infty.
\end{equation*}
In the case of the alternative the reasoning is the same.
This implies (\ref{r2}) and (\ref{r3}).\end{proof}

\begin{proof}[ Proof of Theorem \ref{b4}] By (\ref{x1}), we get
\begin{equation}\label{x5}\begin{split}&
\sup_{P \in \Theta_0} P(K_n =1\,\, \mbox{for some}\,\, n> k) \le \sum_{n=k}^\infty \sup_{P \in \Theta_0} P(K_n =1 )\\& \le \sum_{n=k}^\infty \exp\{-cn\} \le C\exp\{-ck\}.
\end{split}
\end{equation}
The proof of (\ref{x3}) is similar.\end{proof}

\begin{proof}[Proof of Theorem \ref{b4a}] It suffices to show that there is sequence of
tests $K_n$ such  that
\begin{equation}\label{r2a}
\sum_{n=1}^\infty \alpha(K_n) < \infty.
\end{equation}
The tests $K_n$ defined in the proof of Theorem \ref{b3} with $\Theta_{0i}= \Theta_0$ satisfy (\ref{r2a}).\end{proof}
\begin{proof}[ Proof of Theorem \ref{5a} {\sl ii}]
Suppose the contrary. Then there are sequences $\tau_k \in \Theta_0$ and $\eta_k \in \Theta_1$ such that $\rho(\tau_k,\eta_k) \to 0$ as $k \to \infty$. Since $\Theta_0$ is relatively compact, there is $\tau \in \Psi$ and a subsequence $\tau_{k_l}\in \Theta_0$ such that $\tau_{k_l} \to \tau$ as $l \to \infty$. Then $\eta_{k_l} \to \tau$ as $l \to \infty$. Since $\Theta_0$ and $\Theta_1$ are saturated, by Theorem \ref{seq1}, we have a contradiction.
\end{proof}
\begin{proof}[Proof of Theorem \ref{b6w}] Suppose the functions $h_i, 1 \le i \le d$ are bounded. Define the signed measures $H_i, 1 \le i \le d,$ with the densities $\frac{dH_i}{dP} = h_i$. Define the set $$\Upsilon = \left\{G: G = \sum_{i=1}^d b_i H_i,  P + G \in \Lambda, b_i \in R^1, 1 \le i \le d \right\}.$$ On the set $\Upsilon$ the condition $A$ coincides with the definition of Hadamard derivative. Since Hadamard differentiablity implies uniform differentiability on compacts ( Shapiro \cite{sh}, Prop 3.3), we get
\begin{equation*}
\left| T(P + t_n G) - T(P) - t_n\int \vec h d G \right| = o(|t_n|) \,\,\,\,\,\mbox{as}\,\,\,\,\, t_n \to 0,\,\,\,\,n \to \infty
\end{equation*}
uniformly on $$ G \in \left\{H: H =\sum_{i=1}^d b_i H_i, |b_i| \le 1, 1 \le i \le d\right\}.$$

Hence we easily get that $T$ is saturated.

If $\vec h$ is unbounded, we can approximate $\vec h$ bounded functions and repeat the same reasoning.
\end{proof}
\begin{proof}[ Proof of Theorem \ref{si1} ii]
For any $S_1,S_2 \in L_2(0,1)$ denote
$$
(S_1,S_2) = \int_0^1 S_1(t) S_2(t) dt.
$$
We write the problem of signal detection in the coordinates of some orthonormal basis $\{\phi_j\}_{j=1}^\infty$. In this setting we observe Gaussian random vector
$$
y =y_\epsilon=\{y_j\}_{j=1}^\infty,\quad y_j = s_j+ \epsilon\xi_j,\quad \xi_j \sim N(0,1)
$$
with $s_j= (S,\phi_j)$.

Denote $s = \{s_j\}_{j=1}^\infty$ and $\xi= \{\xi_j\}_{j=1}^\infty$.

For $i=0,1,$ we define the sets
$$
\Psi_i = \{s: s=\{s_j\}_{j=1}^\infty, s_j= (S,\phi_j), S \in \Theta_i\}.
$$
Define the linear operator $A: l_2 \to l_2$ such that, for any $s \in l_2$, there holds $As = u$, where $u=\{u_j\}_{j=1}^\infty, u_j = s_j/j, 1 \le j < \infty$.

Define the sets $\Upsilon_i = \{u: u = As, s \in \Psi_i\}$.

Probability measures $\mu_\epsilon$ of random vectors $Ay_\epsilon$  are Radon measures and are tight in  $l_2$. Therefore we can implement Theorem \ref{tg1}.  Relative compactness of sets $\Upsilon_i, i=0,1,$ having disjoint closures, imply that there are finite number of functionals $f_1,\ldots,f_k$ and $\epsilon > 0$ such that for any $u_0=\{u_{0j}\}_{j=1}^\infty \in \Upsilon_0$ and for any $u_1=\{u_{1j}\}_{j=1}^\infty \in \Upsilon_1 $ there is functional $f_i = \{f_{ij}\}_{j=1}^\infty, 1 \le i \le k$, such that
\begin{equation}\label{zz1}
\sum_{j=1}^\infty f_{ij}  (u_{1j}- u_{0j}) > \epsilon.
\end{equation}
This implies that there is $m$ such that for all $u_0 \in \Upsilon_0$ and $u_1 \in \Upsilon_1 $
there holds
\begin{equation}\label{zz2}
\sum_{j=1}^m f_{ij}  (u_{1j}- u_{0j}) > \epsilon/2.
\end{equation}
Denote $l_{2m}$ linear subspace of $l_2$ generated by   $m$ first unit vectors of system of coordinates in $l_2$.

For  $i=0,1$ define the sets $\Upsilon_{im} = \Pi_m \Upsilon_{i}$ where $\Pi_m$ is the projection operator on
 $l_{2m}$.

Define the subset $\Gamma_m = A^{-1} l_{2m}$.

There holds $A^{-1} \Upsilon_{im} = \Pi_{\Gamma_m} \Psi_i$, where $\Pi_{\Gamma_m}$ is projection operator on $\Gamma_m$.

The closures of sets $A^{-1} \Upsilon_{0m}$ and  $A^{-1} \Upsilon_{1m}$ in the topology defined on $\Gamma_m$ are bounded and disjoint that implies Theorem statement.

\end{proof}

\end{document}